\documentclass[BCOR=12mm, DIV=default,twoside, pagesize, draft]{scrartcl}

\addtokomafont{caption}{\small}
\setkomafont{captionlabel}{\sffamily\bfseries}
 
\usepackage{tikz}     
\usetikzlibrary{arrows,automata}

\usepackage{amsmath, amssymb, amsthm} 
\usepackage[utf8x]{inputenc}
\usepackage{libertine} 
\usepackage[T1]{fontenc}
\usepackage[english]{babel} 
\usepackage{url}  
\usepackage{faktor,csquotes,color} 
\usepackage{pgfplots} 
\usepackage{enumitem}

		\usepackage{float} 
		\usetikzlibrary{arrows}
		\usetikzlibrary{datavisualization.formats.functions}
		\pgfplotsset{compat=1.12}
		\usepgfplotslibrary{fillbetween}
		\usetikzlibrary{patterns}

\usepackage{changes}

\definechangesauthor[name={Kristoffer}, color=blue]{kri}
\newcommand{\stkout}[1]{\ifmmode\text{\sout{\ensuremath{#1}}}\else\sout{#1}\fi}
\setdeletedmarkup{\stkout{#1}}

\usepackage{units,varioref}

\usepackage[shortcuts]{extdash}
\usepackage{setspace,scrpage2}
\usepackage{mathtools}
\mathtoolsset{showonlyrefs,showmanualtags}
\numberwithin{equation}{section}

\allowdisplaybreaks 
\clearscrheadfoot
\automark[section]{section}
\ohead{\headmark}
\ofoot[\pagemark]{\pagemark}

\newtheoremstyle{break}{\topsep}{\topsep}{\itshape}{}{\bfseries}{.}{\newline}{}
\newtheoremstyle{exampl}{\topsep}{\topsep}{\upshape}{}{\bfseries}{.}{\newline}{}

\theoremstyle{plain}
\newtheorem{thm}{Theorem}[section]
\newtheorem{lem}[thm]{Lemma}  
  
\newtheorem{cor}[thm]{Corollary}

\theoremstyle{break}

\theoremstyle{definition}
\newtheorem{defi}[thm]{Definition}

\theoremstyle{exampl}

\theoremstyle{remark}
\newtheorem{rem}[thm]{Remark}

\DeclareMathOperator{\E}{\mathbb E}

\definecolor{mygray}{gray}{.5}

\title{Optimal dividends and capital injection under dividend restrictions}
\author{
Kristoffer Lindensj\"o\footnote{Department of Mathematics, Stockholm University, Sweden. kristoffer.lindensjo@math.su.se.} 
\and Filip Lindskog\footnote{Department of Mathematics, Stockholm University, Sweden. lindskog@math.su.se.}}

\begin{document}

\maketitle

\begin{abstract}
We study a singular stochastic control problem faced by the owner of an insurance company that dynamically pays dividends and raises capital in the presence of the restriction that the surplus process must be above a given 
\emph{dividend payout barrier}  in order for dividend payments to be allowed. 
Bankruptcy occurs if the surplus process becomes negative and there are proportional costs for capital injection. 
We show that one of the following strategies is optimal: 
(i) Pay dividends and inject capital in order to reflect the surplus process at an upper barrier and at $0$, implying bankruptcy never occurs.
(ii) Pay dividends in order to reflect the surplus process at an upper barrier and never inject capital --- corresponding to absorption at $0$ --- implying bankruptcy occurs the first time the surplus reaches zero.
We show that if the costs of capital injection are \emph{low}, then a sufficiently high dividend payout barrier will change the optimal strategy 
from type (i) (without bankruptcy) 
to type (ii) (with bankruptcy).  
Moreover, if the costs are \emph{high}, then the optimal strategy is of type (ii) regardless of the dividend payout barrier. 
The uncontrolled surplus process is a Wiener process with drift. 
\end{abstract}



\section{Introduction} \label{intro}
Insurance risk was originally studied in terms of ruin probability. %
However, this approach may underestimate risk since insurance companies are realistically more interested in maximizing company value than minimizing risk and an alternative approach is therefore to study optimal dividend policies 
--- in the sense of maximizing the expected value of the sum of discounted future dividend payments --- 
as suggested by De Finetti in the 1950s. 
A vast literature on various versions of the classical optimal dividend problem has since emerged. Some of the more common modeling choices regard for example:
\begin{itemize}
\item The dynamics of the (uncontrolled) surplus process; e.g. the classical Cramér-Lundberg model, the dual model, or an It\^{o} diffusion model.
\item If capital injection is allowed or not.
\item If bankruptcy is allowed or not. Bankruptcy not being allowed corresponds to obligatory capital injection to avoid bankruptcy.
\item If the insurance company is regulated; e.g. by capital requirements or dividend restrictions.
\item Different kinds of market frictions; e.g. fixed or proportional costs for capital injection.
\item Additional decision variables for the insurance company; e.g. reinsurance level or investment policy.
\end{itemize}
Corporations in financial and insurance markets in the real world typically have the possibility of both going bankrupt and raising equity capital from its owners (capital injection). One of the first papers to take both of these characteristics into account simultaneously is \cite{lokka2008optimal} which studies a singular stochastic control problem corresponding to the optimal dividend problem with the possibility of both capital injection and bankruptcy, under the assumption that the uncontrolled surplus process is a Wiener process with drift. The authors find that depending on the parameters of the model it is either optimal to pay dividends in order to reflect the surplus process at an upper barrier and never inject capital, or to pay dividends and inject capital in order to reflect the surplus process at an upper barrier and at $0$.

However, corporations in financial and insurance markets are also regulated. In order to take this characteristic into account \cite{paulsen2003optimal} studies the optimal dividend problem in a model with solvency constraints, meaning that it is not allowed to pay dividends unless the surplus process exceeds a given constant --- in the present paper called \emph{dividend payout barrier}. Capital injection is not considered. The author finds that it is optimal to use a reflection strategy with the barrier being the maximum of the dividend payout barrier and the reflection barrier that would have been optimal without regulation. The uncontrolled surplus process is a general It\^{o} diffusion.

The main contribution and objective of the present paper is to consider all three mentioned characteristics by studying a singular stochastic control problem which allows capital injection as well as bankruptcy under regulation of the dividend payout barrier type. 

Section \ref{prev-lit} contains a survey of some of the related literature. 
In Section \ref{setup} we formulate the main problem. 
In Section \ref{restrict} we formulate and solve two problems which are intimately connected to the main problem. 
In Section \ref{sec:solution} we use the results of Section \ref{restrict} to solve the main problem; 
the main result is Theorem \ref{mainTHM}.  Section \ref{graphs-sec} contains graphical illustrations. Section \ref{sec-conclusions} contains conclusions and ideas for future research.

\subsection{Previous literature} \label{prev-lit}
This section contains a brief survey of some of the literature on the optimal dividend problem related to the present paper. Further comparisons of the present paper to some of the references are made in the sections below.  
More complete surveys are \cite{albrecher2009optimality,avanzi2009strategies,taksar2000optimal}; see also \cite{alvarez2018class}.

The results of \cite{lokka2008optimal} (see the previous section) are in \cite{zhu2016optimal} extended to a general It\^{o} diffusion model with a growth restriction for the drift function. The results of \cite{paulsen2003optimal} (see the previous section) are in \cite{he2008optimal} extended by the introduction of the possibility of reinsurance.

In \cite{he2009optimal} both fixed and proportional transaction costs for capital injection, as well as reinsurance and bankruptcy are considered; the underlying model is a Wiener process with drift and there is no regulation. 
Other papers studying different models with capital injection and bankruptcy without regulation are 
\cite{avanzi2011optimal,dai2010optimal,zhu2016optimal}. %
In \cite{zhou2012optimal} proportional reinsurance and a maximum dividend rate restriction are studied in a particular diffusion model; capital injection and the possibility of bankruptcy are studied separately.

In \cite{kulenko2008optimal} the classical Cramér–Lundberg model without bankruptcy is studied. 
A similar model with solvency constraints is studied in \cite{zhang2010optimal}. 
In \cite{avram2007optimal} a spectrally negative Lévy process model without bankruptcy is considered.
Other papers studying different models without bankruptcy are \cite{paulsen2008optimal,peng2012optimal,schmidli2017capital,sethi2002optimal,yao2011optimal}.

In \cite{bai2012optimal} an It\^{o} diffusion model with fixed transaction costs and solvency constraints without capital injection is studied. 
In \cite{de2017dividend,grandits2013optimal} %
the dividend problem without capital injection is considered for the Wiener process with drift and a finite time horizon.

In \cite{chen2016optimal} an optimal dividend problem given time-inconsistent preferences is studied using the game-theoretic approach to time-inconsistent stochastic control (general references for this approach include \cite{tomas-continpubl,christensen2017finding,christensen2018time,lindensjo2017timeinconHJB}).

\section{Problem formulation and preliminaries} \label{setup}
Consider a filtered probability space $(\Omega,{\cal F},\mathbb{P},\underline{\mathcal{F}})$ satisfying the usual conditions and supporting a Wiener process $W$. The controlled surplus process $X$ of an insurance company is given by
\begin{align}
X _t & = x + \mu t + \sigma W_t+C_t - D_t, \label{X-SDE} 
 \end{align}
where the process $D$ corresponds to accumulated dividends to the owner of the insurance company and the process $C$ corresponds to accumulated capital injection from the owner. The initial surplus satisfies $x\geq 0$ and the parameters satisfy $\mu>0$ and $\sigma>0$.
We suppose that for a given financing strategy $(C,D)$ the value of the insurance company is the expected value 
of the sum of the discounted future cash flow to the owner and that the owner wants to maximize the value of the insurance company. In mathematical terms we thus %
consider the singular stochastic control problem
\begin{align}   
V(x;b_r) &:= \sup_{(C,D)\in \mathcal A(x,b_r)}\mathbb{E}_x\left( 
\limsup_{t\rightarrow \infty}\left(
\int_0^{\tau \wedge t}e^{-\alpha s}dD_s
-k\int_0^{\tau \wedge t}e^{-\alpha s}dC_s 
\right)\right), \label{the-problem}\\
\tau  &:= \inf\{t\geq 0: X_t  <0\}, \label{tau-D}
\end{align}
where we interpret
$k>1$ as a proportional cost of injecting capital (equity issuance costs), 
$\alpha>0$ as a discount factor, 
$\tau$ as the random bankruptcy time 
and where $\mathcal A(x,b_r)$ is the set of admissible strategies:

\begin{defi}\label{admissible} For a given initial surplus $x\geq 0$ and \emph{dividend payout barrier} $b_r\geq 0$ a pair $(C,D)$ is said to be an \emph{admissible} strategy if $C$ and $D$ are non-decreasing LCRL $\underline{\mathcal{F}}$-adapted processes with $C_0=D_0=0$ satisfying the \emph{dividend payout condition}
\begin{align}   
\int_0^\tau I_{\{X_t <b_r\}}dD_t = 0 \enskip  \mbox{a.s.}\label{Admiss1} 
\end{align}
\end{defi} 

The main objective of the present paper is to study problem \eqref{the-problem}. This problem has according to the authors' knowledge not been studied before.

Obviously the parameters of the model are such that either condition \eqref{lemma1:newcondition-k} below holds, or 
\eqref{lemma1:newcondition-k} holds with reversed inequality, which is interpreted as proportional costs of capital injection $k$ being 
\emph{low} or \emph{high}, respectively. 
In Theorem \ref{mainTHM} we will see that which is the case determines the kind of solution problem \eqref{the-problem} has. 
\begin{align}
k & \leq \frac
{ r_1-r_2
}{ 
  r_1 \left(\frac{r_2^2}{r_1^2}\right)^{\frac{r_1}{r_1-r_2}}
- r_2 \left(\frac{r_2^2}{r_1^2}\right)^{\frac{r_2}{r_1-r_2}}
} \label{lemma1:newcondition-k},\\
& \mbox{where } \enskip r_1 := -\frac{\mu}{\sigma^2} +\sqrt{\frac{\mu^2}{\sigma^4}+\frac{2\alpha}{\sigma^2}}, 
\enskip r_2 := -\frac{\mu}{\sigma^2} -\sqrt{\frac{\mu^2}{\sigma^4}+\frac{2\alpha}{\sigma^2}}. \label{def:r1r2}\\
& \mbox{Note that $r_2< 0 < r_1$ and $r_2^2>r_1^2$.} \label{properties:r1r2}
\end{align}

\begin{rem} \label{D-plus-remark} Formally we write $dD_t=dD^c_t + \Delta D_t$ where $D^c$ denotes the continuous part of $D$ while $\Delta D_t:= D_{t+}-D_t$, for $D_{t+}:= \lim_{h\searrow 0}D_{t+h}$, denotes jumps in $D$. The processes $C$ and $X$ are treated analogously.
\end{rem}
 
\begin{rem} \label{cap-req-rem}  Condition \eqref{Admiss1} implies that if the surplus at time $t$ is smaller than the  dividend payout barrier $b_r$ then dividends are not allowed, i.e. $dD_t = 0$. %
Thus, if $X_{t+}<b_r$ then a potential jump $\Delta X_t$ cannot have been caused by $D$ and must have been caused by $C$ and hence $X_t\leq X_{t+}<b_r$, implying that $dD_t = 0$ also in this case; in mathematical terms this means that \eqref{Admiss1} implies that
\begin{align}   
\int_0^\tau I_{\{X_{t+} <b_r\}}dD_t = 0 \enskip  \mbox{a.s.}\label{Admiss2} 
\end{align}
In other words, the surplus directly \emph{after} a dividend payment cannot be lower than the dividend payout barrier $b_r$ either.
\end{rem}

\begin{rem} An inequality analogous to that of \eqref{lemma1:newcondition-k} is used in \cite{lokka2008optimal} when studying problem \eqref{the-problem} without the presence of a dividend payout barrier, i.e. with $b_r=0$. Regulation of the type \eqref{Admiss1} was first studied in \cite{paulsen2003optimal}.
\end{rem}

\subsection{Preliminaries} \label{Preliminaries}
Here we informally recall well-known results that are used throughout the present paper; cf. e.g. 
\cite{Karatzas2,pilipenko2014introduction,shreve1984optimal}. 
%
%
%
%
For any $b>0$ and $x \in [0,b]$ there exists an $\underline{\mathcal{F}}$-adapted non-decreasing continuous processes $\bar{D}^b$ with $\bar{D}^b_0=0$ such that the process $X$ defined by,
\begin{align}
X _t= x + \mu t + \sigma W_t - \bar{D}^b_t, \label{X-SDEref}
\end{align}
is reflected at $b$ and satisfies,
\begin{align}
dX _t = \mu dt + \sigma dW_t, \enskip \mbox{when $X_t<b$},
\end{align}
with $\bar{D}^b$ being constant on any interval where $X_t<b$.
There also exists a pair of $\underline{\mathcal{F}}$-adapted non-decreasing continuous processes $(C^0,D^b)$ with $C^0_0=D^b_0=0$ such that the process $X$ defined by,
\begin{align}
X _t= x + \mu t + \sigma W_t + C^0_t - D^b_t,  \label{X-SDErefref}
\end{align}
is reflected at $b$ and $0$, and satisfies
\begin{align}
dX _t = \mu dt + \sigma dW_t, \enskip \mbox{when $0<X_t<b$},
\end{align}
with $D^b$ being constant on any interval where $X_t<b$ and $C^b$ being constant on any interval where $X_t>0$. 
In the case $x>b$, $\bar{D}^b$ and $D^b$ are defined so that the  corresponding processes in 
\eqref{X-SDEref} and \eqref{X-SDErefref} jump from $x$ to $b$ at $t=0$.

The pair $(C^0,D^b)$ is in the present paper said to be a \emph{double barrier strategy}, 
while $(0,\bar{D}^b)$, or simply $\bar{D}^b$, is said to be an \emph{upper barrier strategy}. For an upper barrier strategy $\bar{D}^b$ and $\tau$ defined in \eqref{tau-D}, the value function 
$x \mapsto \mathbb{E}_x\left(\int_0^{\tau } e^{-\alpha t}d\bar{D}_t^b\right)$ is the unique solution to,
\begin{align}
\alpha f(x) & = \mu f'(x) +\frac{1}{2} \sigma^2 f''(x), \enskip 0 \leq x < b, \label{ODE}\\ 
f(x) & = x-b+ f(b), \enskip x \geq b, \label{ODEcondnodiv0} \\ 
f(0)&=0, \enskip f'(b) = 1, \label{ODEcondnodiv1}
\end{align}
cf. \cite[Lem. 2.1, Cor. 2.2 \& Ex. 1]{shreve1984optimal}. The general solution to the ODE \eqref{ODE} is, for $r_1,r_2$ defined in \eqref{def:r1r2} and constants $c_1,c_2$,
\begin{align} 
f(x) & = c_1 e^{r_1x} + c_2 e^{r_2x}. \label{gen-sol}
\end{align}
The general solution together with \eqref{ODEcondnodiv0} and the boundary conditions \eqref{ODEcondnodiv1}, and simple calculations, yield, for any $b>0$,
\begin{equation}
\mathbb{E}_x\left(\int_0^{\tau } e^{-\alpha t}d\bar{D}_t^b\right)=\begin{cases}
\frac{e^{r_1 x}-e^{r_2 x}}{r_1e^{r_1 b}-r_2e^{r_2 b}}, 					&\ 0 \leq x \leq b,\\
x-b+ \frac{e^{r_1 b}-e^{r_2 b}}{r_1e^{r_1 b}-r_2e^{r_2 b}} ,		& x>b. \label{ref-abs-valF}
	\end{cases}
	\end{equation}

For a double barrier strategy $(C^0,D^b)$ the stopping time in \eqref{tau-D} trivially satisfies $\tau=\infty$ a.s. The value function 
$x \mapsto \mathbb{E}_x\left(
\int_0^{\infty} e^{-\alpha t}dD_t^b
-k\int_0^{\infty} e^{-\alpha t}dC^0_t
\right)$ with $k>1$ is well-defined and is the unique solution to \eqref{ODE} together with \eqref{ODEcondnodiv0} and the boundary conditions,
\begin{align}
f'(0)=k, \enskip f'(b) = 1, \label{ODEcondnobank1}
\end{align}
cf. \cite[Lem. 2.1, Cor. 2.2 \& Ex. 1]{shreve1984optimal}. 
The general solution \eqref{gen-sol} together with \eqref{ODEcondnodiv0} and \eqref{ODEcondnobank1} yield, 
for any $b>0$, 
\begin{align}
\mathbb{E}_x\left(
\int_0^{\infty} e^{-\alpha t}dD_t^b
-k\int_0^{\infty} e^{-\alpha t}dC^0_t
\right) 
\quad\quad\quad\quad\quad\quad\quad\quad\quad\quad\quad\quad\quad\quad 
\end{align}
\begin{equation}
=\begin{cases}
\frac{1}{e^{r_1b}-e^{r_2b}}\left(\frac{1-ke^{r_2b}}{r_1}e^{r_1 x}-\frac{1-ke^{r_1b}}{r_2}e^{r_2x}\right),  &\ 0 \leq x \leq b,\\
x-b+  
\frac{1}{e^{r_1b}-e^{r_2b}}\left(\frac{1-ke^{r_2b}}{r_1}e^{r_1 b}-\frac{1-ke^{r_1b}}{r_2}e^{r_2b}\right),		& x>b.
\label{ref-ref-valF}
\end{cases}
\end{equation}
 
\begin{rem} The process $\bar{D}^b$ can be pathwise defined as $\bar{D}^b_t=\max_{0\leq s\leq t}(x+\mu s+ \sigma W_s-b)_+$ which can be seen using the corresponding Skorohod equation, cf. \cite[Section 3.6 C]{Karatzas2} and \cite{asmussen1997controlled}. 
The pair $(C^0,D^b)$ can be constructed pathwise in a procedure involving iteratively using the solutions to the Skorohod equations for reflection at $b$ and at $0$ in turn, as noted in \cite[p. 960]{lokka2008optimal}.
\end{rem}

\section{Two restricted problems}\label{restrict}
In order to solve problem \eqref{the-problem} it is useful to first study two related problems for which the set of admissible strategies is further restricted.

\subsection{Capital injection not allowed}
Here we consider problem \eqref{the-problem} under the additional restriction that capital injection is not allowed 
--- this problem has been studied in the literature, see \cite{paulsen2003optimal}, and we here only recall the solution. That is, we consider admissible strategies $(C,D) \in \mathcal A(x,b_r)$ for which
\begin{align} 
C_t =0 \enskip \mbox{for all $t \geq 0$ a.s.} \label{no-div-assum}
\end{align}
For this restricted problem we can write, using also that $D$ is non-decreasing, the optimal value function \eqref{the-problem} as,
\begin{align}
\sup_{(0,D)\in \mathcal A(x,b_r)}\mathbb{E}_x\left(\int_0^{\tau } e^{-\alpha t}dD_t\right).\label{the-problem-nodiv}
\end{align}
The solution to this problem is given by:
\begin{thm} \label{paulsen-thm} The upper barrier strategy  $\bar{D}^{b}$ with $b=b_r\vee b^{*}$ is optimal in \eqref{the-problem-nodiv}, where
\begin{align} 
b^*:= \frac{\log\left(r_2^2/r_1^2\right)}{r_1-r_2}>0.\label{b*}
\end{align}
\end{thm}
\begin{proof} 
The result can be proved using arguments similar to those in the proof of Theorem \ref{mainTHM}. The constant in \eqref{b*} is also found in \cite[Eq. (5.6)]{shreve1984optimal} and \cite[Eq. (3.6)]{lokka2008optimal}, see also Remark \ref{prev-lit-rem1}. Note that $b^*>0$ by \eqref{properties:r1r2}. The result also follows from \cite[Theorem 2.2]{paulsen2003optimal}. 
\end{proof}

The value function corresponding to the strategy in Theorem \ref{paulsen-thm} can, using the results of Section \ref{Preliminaries}, be written as,
\begin{equation}
G(x;b_r):=\begin{cases}
\frac{e^{r_1 x}-e^{r_2 x}}{r_1e^{r_1 b}-r_2e^{r_2 b}}, 					&\ 0 \leq x \leq b,\\
x-b+ \frac{e^{r_1 b}-e^{r_2 b}}{r_1e^{r_1 b}-r_2e^{r_2 b}} ,		& x>b, \enskip \mbox{ where } b=b_r\vee b^{*}. \label{G-def}
	\end{cases}
	\end{equation}
We will usually write $G(x)$ instead of $G(x;b_r)$ for convenience. Lemma \ref{lemmaVA} presents properties of the function $G$ that are used when solving problem \eqref{the-problem} in Section \ref{sec:solution}.

\begin{lem}  \label{lemmaVA} 
For  $G$ defined in \eqref{G-def} holds:
\begin{enumerate}[label=(\Roman*)]
\item \label{lemmaVA:item1} $G'(x) > 0$ for all $x\geq 0$.
\item \label{lemmaVA:item2} If condition \eqref{lemma1:newcondition-k}  holds with reversed inequality, then $G'(0) \leq k$.

\end{enumerate}
\end{lem}

\begin{proof} We use \eqref{properties:r1r2} repeatedly. Item \ref{lemmaVA:item1} is directly verified. Let us prove \ref{lemmaVA:item2}. We find 
\begin{align} 
G'(0) = \frac{r_1-r_2}{r_1e^{r_1b} -r_2e^{r_2b} }. 
\end{align}
From \eqref{b*} follows $e^{b^*} = \left(\frac{r_2^2}{r_1^2}\right)^{\frac{1}{r_1-r_2}}$. In the case $b = b^{*}$ (i.e. $b_r\leq b^*$), 
\begin{align} 
G'(0) & = \frac{r_1-r_2}{r_1e^{r_1b^*} -r_2e^{r_2b^*}} \label{help-1}\\
& =  \frac{ r_1-r_2}{   r_1 \left(\frac{r_2^2}{r_1^2}\right)^{\frac{r_1}{r_1-r_2}}- r_2 \left(\frac{r_2^2}{r_1^2}\right)^{\frac{r_2}{r_1-r_2}}}.
\end{align}
Thus, $G'(0) \leq k$ if and only if \eqref{lemma1:newcondition-k} holds with reversed inequality in the case $b = b^{*}$. 
Let us view $G'(0)$ as a function of $b$ which we denote by $h$, i.e. let
\begin{align} 
h(b):= \frac{r_1-r_2}{r_1e^{r_1b} -r_2e^{r_2b}}.\label{h-function}
\end{align}
From the definition of $b^*$ follows that the derivative of the denominator in \eqref{h-function},
i.e. $r_1^2e^{r_1b} -r_2^2e^{r_2b}$, 
is (strictly) positive when $b>b^*$ and (strictly) negative when $b<b^*$. It follows that $h(b)$ is (strictly) decreasing in  $b$ for $b>b^*$ and (strictly) increasing in $b$ for $b<b^*$. Hence, $h(b)$ is maximal at $b^*$. These facts give $G'(0)\leq k$ for all $b$. We remark that these facts and the function $h$ will be used below.
\end{proof}

\begin{rem} \label{prev-lit-rem1} Problem \eqref{the-problem-nodiv} was for a general It\^{o} diffusion solved in \cite{paulsen2003optimal}. 
In the case without a dividend payout barrier, i.e. with $b_r=0$, the problem \eqref{the-problem-nodiv} is the 
well-known absorption problem first studied, for a general It\^{o} diffusion, in \cite{shreve1984optimal}. 
In particular, in \cite[Theorem 4.3]{shreve1984optimal} it was shown --- under appropriate assumptions --- that an upper barrier strategy $\bar{D}^{b^*}$ is optimal, where the barrier $b^*$ is determined by the additional boundary condition 
\begin{align} 
f''(b^*)=0 \label{ODEcondnodiv2}.
\end{align}
(If no such $b^*$ exists, then no optimal strategy exists and the optimal value function is $\lim_{b\rightarrow \infty}\mathbb{E}_x\left(\int_0^{\tau } e^{-\alpha t}d\bar{D}_t^b\right)$.) In particular, $b^*$ in Theorem \ref{paulsen-thm}  is found using condition \eqref{ODEcondnodiv2} for the value function \eqref{ref-abs-valF}.
\end{rem}

\subsection{Bankruptcy not allowed}
Here we consider the problem \eqref{the-problem} under the restriction that bankruptcy is not allowed. That is, we consider admissible strategies  $(C,D) \in \mathcal A(x,b_r)$ for which,
\begin{align}  \label{no-bankrup-assum} 
X_t \geq 0 \enskip \mbox{for all $t \geq 0$ a.s.}
\end{align}
We denote the set of such strategies by $\mathcal{A}^R(x,b_r)$. For this restricted problem holds $\tau=\infty$ a.s. and the optimal value function \eqref{the-problem} can be written as,
\begin{align} 
\sup_{(C,D)\in \mathcal A^R(x,b_r)}\mathbb{E}_x\left(
\limsup_{t\rightarrow \infty}
\left(
\int_0^te^{-\alpha s}dD_s - k\int_0^te^{-\alpha s}dC_s
\right)\right).
\label{the-problem-no-bankrup} 
\end{align}
Problem \eqref{the-problem-no-bankrup} has according to the authors' knowledge not been considered before. The solution is given by:

\begin{thm}\label{newthm1} The double barrier strategy  $(C^0,D^{b})$ with $b=b_r\vee b^{**}$ is optimal in \eqref{the-problem-no-bankrup}, where $b^{**}>0$ is the unique positive solution to the equation,
\begin{align} 
& r_1 e^{-r_2b^{**}} - r_2 e^{-r_1b^{**}}  = k(r_1-r_2). \label{b**}
\end{align}
\end{thm}
\begin{proof} 
The result can proved using arguments analogous to those in the proof of Theorem \ref{mainTHM}. 
The uniqueness of $b^{**}$ is verified by noting that $r_1e^{-r_2b} - r_2 e^{-r_1b}$ is strictly increasing in $b$; to see this use differentiation and \eqref{properties:r1r2}. It is easy to see that $b^{**}$ must be positive. A proof in the case $b_r=0$ is found in \cite[Sec. 4,5]{lokka2008optimal}.
\end{proof}

The value function corresponding to the strategy in Theorem \ref{newthm1} can, using the results of Section \ref{Preliminaries}, be written as,
\begin{equation}
H(x;b_r):=\begin{cases}
\frac{1}{e^{r_1b}-e^{r_2b}}\left(\frac{1-ke^{r_2b}}{r_1}e^{r_1 x}-\frac{1-ke^{r_1b}}{r_2}e^{r_2x}\right),  &\ 0 \leq x \leq b,\\
x-b+  
\frac{1}{e^{r_1b}-e^{r_2b}}\left(\frac{1-ke^{r_2b}}{r_1}e^{r_1 b}-\frac{1-ke^{r_1b}}{r_2}e^{r_2b}\right),		& x>b,
\enskip b=b_r\vee b^{**}.
 \label{H-def}
\end{cases}
\end{equation}
We will usually write $H(x)$ instead of $H(x;b_r)$. Lemma \ref{lemmaVR} presents properties of the function $H$ that are used when solving the main problem \eqref{the-problem} in Section \ref{sec:solution}. 
The proof of Lemma \ref{lemmaVR} relies on the same type of arguments as the proof of Lemma \ref{lemmaVA} and is found in the appendix.

\begin{lem}  \label{lemmaVR} 
For  $H$ defined in \eqref{H-def} holds:
\begin{enumerate}[label=(\Roman*)] 

\item \label{lemma1:new:item1.5} $H'(x)>0$ for all $x\geq 0$.

 \item \label{lemma1:new:item1} $H(0)\geq 0$ is equivalent to  
\begin{align} \label{lemma1:new:item1eq1}
 r_1e^{r_1  b} - r_2e^{r_2  b} \leq \frac{r_1-r_2}{k}.
\end{align}
Moreover, $H(0)\leq 0$ is equivalent to \eqref{lemma1:new:item1eq1} with reversed inequality.

\item \label{lemma1:new:item2.5} 
Condition \eqref{lemma1:newcondition-k} is equivalent to $b^{**}\leq b^{*}$ which is equivalent to 
\begin{align} \label{newthm3:cond}
r_1e^{r_1b^{**}}-r_2e^{r_2b^{**}} \leq  \frac{r_1-r_2}{k}.%
\end{align}
Moreover, \eqref{lemma1:newcondition-k} with reversed inequality is equivalent to $b^{**}\geq b^{*}$, which is equivalent to \eqref{newthm3:cond} with reversed inequality.

\item \label{lemma1:new:item2} 
Suppose $b_r\leq b^{**}$. 
Then, $H(0)\geq 0$ is equivalent to \eqref{lemma1:newcondition-k}, 
and $H(0) \leq 0$ is equivalent to \eqref{lemma1:newcondition-k} with reversed inequality.

\item \label{lemma1:new:item4} Suppose \eqref{lemma1:newcondition-k} holds with reversed inequality. 
Then, for any $b_r$, $H(0)\leq 0$.

\item \label{lemma1:new:item3} Suppose \eqref{lemma1:newcondition-k} holds. Then, $H(0)\geq 0$ if and only if $b_r\leq \hat b$, where
$\hat b$ is the unique solution, on the domain $[b^{**},\infty)$, to the equation 
\begin{align} 
 r_1e^{r_1 \hat b} - r_2e^{r_2 \hat b} = \frac{r_1-r_2}{k}. \label{b-hat}
\end{align}
\end{enumerate} 
\end{lem}

\begin{rem} \label{prev-lit-rem2} In the case $b_r=0$ the problem \eqref{the-problem-no-bankrup} is the 
well-known reflection problem first studied, for a general It\^{o} diffusion, in \cite{shreve1984optimal}; see also 
\cite[Sec. 5]{shreve1984optimal} where the problem is studied for a Wiener process with drift. %
In particular, in  \cite[Theorem 4.5]{shreve1984optimal} it was for the case $b_r=0$ shown --- under appropriate assumptions 
 --- that the double barrier strategy $(C^0,D^{b^{**}})$ is optimal, where the barrier $b^{**}$ is given by the additional boundary condition
\begin{align}
f''(b^{**})=0 \label{ODEcondnobank2}.
\end{align}
(If no such $b^{**}$ exists, then no optimal strategy exists and the optimal value function is 
$\lim_{b\rightarrow \infty}\mathbb{E}_x\left(
\int_0^{\infty} e^{-\alpha t}dD_t^b
-k\int_0^{\infty} e^{-\alpha t}dC^0_t
\right)$.) In particular, $b^{**}$ in Theorem \ref{newthm1} is found using condition \eqref{ODEcondnobank2} for the value function \eqref{ref-ref-valF}.
\end{rem}

\begin{rem} \label{prev-lit-rem3} The equivalences in \ref{lemma1:new:item2.5} in Lemma \ref{lemmaVR} were in the context of studying problem \eqref{the-problem} without a dividend payout barrier, i.e. with $b_r=0$, derived in \cite{lokka2008optimal}.
\end{rem}

\section{Solution to the main problem} \label{sec:solution}
Since our model is Markovian it is reasonable to conjecture that the optimal strategy for problem \eqref{the-problem} involves either that the owner always saves the insurance company from bankruptcy by injecting capital when the surplus process hits zero, or that the owner never does so. Indeed this is what we find in Theorem \ref{mainTHM}. The results in this section are illustrated in graphs in Section \ref{graphs-sec} and interpreted in Section \ref{sec-conclusions}.

\begin{thm}[Main result] \label{mainTHM} 
Consider 
$b^{*}$ defined in \eqref{b*}, 
$b^{**}$  defined in \eqref{b**}  and 
$\hat b$ defined in \eqref{b-hat}. 
For problem \eqref{the-problem} holds:
\begin{enumerate}[label=(\Roman*)]
\item  Suppose   \eqref{lemma1:newcondition-k} holds. 
\begin{enumerate}[label=(\Roman{enumi}.\alph*)]
\item \label{mainTHMnew:1a} If the dividend payout barrier satisfies $b_r\leq \hat b$ then the double barrier strategy $(C^0,D^b)$ with $b = b_r\vee b^{**}$ is optimal.
The corresponding bankruptcy time, cf. \eqref{tau-D}, satisfies $\tau =\infty$ a.s.
\item \label{mainTHMnew:1b} If $b_r\geq \hat b$ then the upper barrier strategy $\bar{D}^b$ with $b = b_r\vee b^{*}$ is optimal. 
The moments of the corresponding bankruptcy time are finite, i.e. $\mathbb{E}_x\left(\tau^n\right)<\infty$ for all $x\geq0$ and $n$. 
\end{enumerate} 
\item \label{mainTHMnew:2} Suppose   \eqref{lemma1:newcondition-k} holds with reversed inequality. Then the upper barrier strategy $\bar{D}^b$ with $b = b_r\vee b^{*}$ is optimal for any given $b_r$. 
The corresponding bankruptcy time satisfies the same condition as in \ref{mainTHMnew:1b}.
\end{enumerate}
%
\end{thm}

 \begin{rem} A recursive formula for the moments of the bankruptcy times in \ref{mainTHMnew:1b} and \ref{mainTHMnew:2} in Theorem \ref{mainTHM} can be found in \cite{wang2008moments}. 
\end{rem}

From Theorem \ref{mainTHM}, \eqref{G-def} and \eqref{H-def} follows directly:
\begin{cor} \label{corr1} The optimal value function for problem \eqref{the-problem} has the representations,
\begin{equation}
V(x;b_r) =\begin{cases}
H(x;b_r), 					&    \mbox{if \eqref{lemma1:newcondition-k} holds and $b_r \leq \hat b$,}\\
G(x;b_r),						&   \mbox{if \eqref{lemma1:newcondition-k} holds with reverse equality or $b_r \geq \hat b$,} \label{mainTHM:eq-valF}
	\end{cases}
	\end{equation}
\begin{equation}
  = H(x;b_r) \vee G(x;b_r).  \quad\quad\quad\quad\quad\quad\quad\quad\quad\quad\quad\enskip
\end{equation} 
\end{cor}
Note that Corollary \ref{corr1}, \eqref{G-def} and \eqref{H-def} give an explicit expression for the optimal value function of problem \eqref{the-problem}. 

The following properties of the solution of problem \eqref{the-problem} are proved in the appendix:

\begin{cor}  \label{cor1.5} Suppose \eqref{lemma1:newcondition-k} holds with strict inequality. 
If $b_r<\hat b$, then $V(0;b_r) = H(0;b_r) > G(0;b_r) = 0$.  
If $b_r>\hat b$, then $H(0;b_r) < V(0;b_r) = G(0;b_r) = 0$.   %
\end{cor}
\begin{rem} The interpretation of Corollary \ref{cor1.5} is that in the case capital injection costs are low 
(in the sense that \eqref{lemma1:newcondition-k} holds holds with \emph{strict} inequality) holds that: 
if the dividend payout barrier satisfies $b_r<\hat b$ then it is \emph{not} optimal to allow bankruptcy and 
if the dividend payout barrier satisfies $b_r>\hat b$ then it is \emph{not} optimal to save the insurance company from bankruptcy. 		
\end{rem}

\begin{cor} \label{cor2} For any fixed $x>0$ holds:
\begin{enumerate}[label=(\Roman*)]

\item \label{cor2:1} $V(x;b_r)$ is decreasing in $b_r$. In particular,

\begin{enumerate}[label=(\Roman{enumi}.\alph*)]
\item \label{cor2:1a}  If \eqref{lemma1:newcondition-k} holds, 
then $V(x;b_r)$ is independent of $b_r$ for $b_r\leq b^{**}$ and strictly decreasing in $b_r$ for $b_r>b^{**}$.

\item \label{cor2:1b} If \eqref{lemma1:newcondition-k} holds with reversed inequality, 
then $V(x;b_r)$ is independent of $b_r$ for $b_r\leq b^{*}$ and strictly decreasing in $b_r$ for $b_r>b^{*}$.
\end{enumerate}

\item \label{cor2:2} $\lim_{b_r\rightarrow \infty} V(x;b_r) = 0$. 
\end{enumerate}
\end{cor} 

\begin{rem} It is easy to show that the results in Corollary \ref{cor2} hold also in the case $x=0$, with the modification that $V(0;b_r)$ is only \emph{strictly} decreasing in the case
$V(0;b_r)>0$, i.e. when \eqref{lemma1:newcondition-k} holds with strict inequality and $b_r <\hat b$.  
\end{rem}

\begin{proof} (Theorem \ref{mainTHM}.) 
We remark that this proof relies neither on Theorem \ref{paulsen-thm} nor on Theorem \ref{newthm1}.

Let us first deal with the results about the bankruptcy time $\tau$. The result for \ref{mainTHMnew:1a} is trivial since $X$ is reflecting at both $b$ and $0$ in this case. The result for  \ref{mainTHMnew:1b} and \ref{mainTHMnew:2} is contained in \cite{wang2008moments}.

Now consider a function 
$g \in \mathcal C^1([0,\infty)) \cap \mathcal  C^2([0,b) \cup (b,\infty))$ for some $b>0$, an arbitrary strategy $(C,D)\in \mathcal A(x,b_r)$ and an arbitrary time $t>0$.  
Similarly to e.g. \cite[p. 60]{shreve1984optimal} and \cite[p. 959]{lokka2008optimal} we note 
that for the right-continous process $(X_{(\tau \wedge t)+})_{t\geq 0}$ (cf. Remark \ref{D-plus-remark}) holds, 
by the It\^{o}-Tanaka formula, that 
\begin{align} 
& e^{-\alpha \tau \wedge t}g(X_{(\tau \wedge t)+})\\
&= g(x)+ 
\int_0^{\tau \wedge t} e^{-\alpha s}\left(\mu g'(X_s) +\frac{1}{2} \sigma^2 g''(X_s)-\alpha g(X_s)\right)I_{\{X_s \neq b\}}ds \\
& \quad + \int_0^{\tau \wedge t} e^{-\alpha s} \sigma g'(X_s) dW_s 
+ \int_0^{\tau \wedge t} e^{-\alpha s} g'(X_{s})dC^c_s - \int_0^{\tau \wedge t} e^{-\alpha s} g'(X_{s}) dD^c_s\\
& \quad + \sum_{0\leq s\leq \tau \wedge t} e^{-\alpha s} (g(X_s + \Delta C_s)- g(X_{s}))
  + \sum_{0\leq s\leq \tau \wedge t} e^{-\alpha s} (g(X_s - \Delta D_s)- g(X_{s})).
\end{align}
The fundamental theorem of calculus gives
\begin{align}
 g(X_s + \Delta C_s)- g(X_{s})  
& =   \int_0^{\Delta C_s}g'(X_s+z)dz,\\ 
 g(X_s - \Delta D_s)- g(X_{s})
& = - \int_0^{\Delta D_s}g'(X_s-z)dz.
\end{align}
Now suppose $g$ is either  
the value function of the upper barrier strategy $\bar{D}^b$ with $b = b_r\vee b^{*}$, given by $G$ in \eqref{G-def}, or 
the value function of the double barrier strategy $(C^0,D^b)$ with $b = b_r\vee b^{**}$, given by $H$ in \eqref{H-def} --- the differentiability condition used above is directly verified in both cases.  
Then $g$ satisfies 
\eqref{ODE} and 
\eqref{ODEcondnodiv0} implying $g'(x) = 1$ for $x\geq b$ and $g''(x)=0$ for $x>b$.

These observations imply that
\begin{align} 
g(x) 
& = e^{-\alpha \tau \wedge t}g(X_{(\tau \wedge t)+}) -
\int_0^{\tau \wedge t} e^{-\alpha s}\left(\mu -\alpha g(X_s)\right)I_{\{X_s > b\}}ds \\
& \quad - \int_0^{\tau \wedge t} e^{-\alpha s} \sigma g'(X_s) dW_s 
- \int_0^{\tau \wedge t} e^{-\alpha s} g'(X_{s})dC^c_s + \int_0^{\tau \wedge t} e^{-\alpha s}g'(X_{s})dD^c_s\\
& \quad - \sum_{0\leq s\leq \tau \wedge t} e^{-\alpha s} \int_0^{\Delta C_s}g'(X_s+z)dz 
 + \sum_{0\leq s\leq \tau \wedge t} e^{-\alpha s} \int_0^{\Delta D_s}g'(X_s-z)dz.
\end{align} 
Lemma \ref{lem1} gives $\left(\mu -\alpha g(X_s)\right)I_{\{X_s > b\}}\leq 0$ and therefore
\begin{align} 
g(x) & \geq e^{-\alpha \tau \wedge t}g(X_{(\tau \wedge t)+}) - \int_0^{\tau \wedge t} e^{-\alpha s} \sigma g'(X_s) dW_s\\ 
&\quad +\int_0^{\tau \wedge t} e^{-\alpha s} g'(X_{s})dD^c_s + \sum_{0\leq s\leq \tau \wedge t} e^{-\alpha s} \int_0^{\Delta D_s}g'(X_s-z)dz\\
&\quad -\int_0^{\tau \wedge t} e^{-\alpha s} g'(X_{s})dC^c_s - \sum_{0\leq s\leq \tau \wedge t} e^{-\alpha s} \int_0^{\Delta C_s}g'(X_s+z)dz.
\end{align} 
Hence,
\begin{align} 
g(x) & \geq  \left(\int_0^{\tau \wedge t} e^{-\alpha s}dD_s- k \int_0^{\tau \wedge t} e^{-\alpha s}dC_s\right)
\label{mainTHM-pf3:eqN0}\\
&\quad +  e^{-\alpha \tau \wedge t}g(X_{(\tau \wedge t)+}) - \int_0^{\tau \wedge t} e^{-\alpha s} \sigma g'(X_s) dW_s 
\label{mainTHM-pf3:eqN1}\\ 
&\quad +\int_0^{\tau \wedge t} e^{-\alpha s} 
(g'(X_{s})-1)dD^c_s + \sum_{0\leq s\leq \tau \wedge t} e^{-\alpha s} \int_0^{\Delta D_s}(g'(X_s-z)-1)dz \label{mainTHM-pf3:eq3}\\
&\quad  + \int_0^{\tau \wedge t} e^{-\alpha s} 
(k-g'(X_{s}))dC^c_s + \sum_{0\leq s\leq \tau \wedge t} e^{-\alpha s} \int_0^{\Delta C_s}(k-g'(X_s+z))dz\label{mainTHM-pf3:eq4}.
\end{align} 
We conclude the proof by considering different cases. 
We rely on $G$ in \eqref{G-def} solving \eqref{ODE} on $[0,b)$, \eqref{ODEcondnodiv0} and \eqref{ODEcondnodiv1} with $b=b_r\vee b^{*}$
and $H$ in \eqref{H-def} solving \eqref{ODE} on $[0,b)$, \eqref{ODEcondnodiv0} and \eqref{ODEcondnobank1} with $b=b_r\vee b^{**}$, cf. Section \ref{Preliminaries}. 

\emph{Case A:} Consider the conditions of \ref{mainTHMnew:1a}. Suppose $g$ in inequality  \eqref{mainTHM-pf3:eqN0}--\eqref{mainTHM-pf3:eq4} is defined as $H$  in \eqref{H-def}. Suppose $b_r\leq b^{**}$ (i.e. $b=b^{**}$). Observe:
\begin{itemize}

\item  
$H'(x) = 1$ for $x\geq b^{**}$ and $H'(0)=k$.
Since $H'(x)>0$ for  $x \geq 0$ (Lemma \ref{lemmaVR}) and  $H''(b^{**}) = 0$ (easily verified) it follows from Lemma \ref{lem2} that $H''(x)<0$ for $x <b^{**}$.  Hence, $H'$ is non-increasing on $[0,b^{**}]$. It follows that $1 \leq H'(x) \leq k$ for $x\geq 0$. We conclude that the expressions \eqref{mainTHM-pf3:eq3} and \eqref{mainTHM-pf3:eq4} are non-negative.

\item Use Lemma \ref{lemmaVR} and $b_r\leq \hat b$ to find $H(0)\geq 0$ and $H'(x)>0$ for $x\geq 0$. Hence, $H(x)\geq 0$ for  $x \geq 0$. We conclude that the first term in \eqref{mainTHM-pf3:eqN1} is non-negative.
	\item If we send $t$ to infinity then the second term  in \eqref{mainTHM-pf3:eqN1} converges a.s. to a random variable with zero expectation (use that $H'(x)$ is a bounded function). 
\end{itemize}
Thus, sending $t$ to infinity ($\limsup$) in \eqref{mainTHM-pf3:eqN0}--\eqref{mainTHM-pf3:eq4} and taking expectation gives
\begin{align} 
H(x) \geq \E_x\left(
\limsup_{t\rightarrow \infty}\left(
\int_0^{\tau \wedge t} e^{-\alpha s} dD_s - k\int_0^{\tau \wedge t} e^{-\alpha s}dC_s
\right)\right).
\end{align}
Since $(C,D) \in \mathcal A (x,b_r)$ was arbitrarily chosen follows that \ref{mainTHMnew:1a} holds in the case $b=b^{**}$ (recalling that $H(x)$ is the value function attained by the strategy in \ref{mainTHMnew:1a}).

Now suppose $b_r>b^{**}$ (i.e. $b=b_r$). Observe:
\begin{itemize}
\item $H'(x)=1$ for $x\geq b_r$. 
The dividend payout condition \eqref{Admiss1} (see also Remark \ref{cap-req-rem}) therefore implies that the expressions in \eqref{mainTHM-pf3:eq3} vanish (either the derivatives are equal to one or there are no dividends).

\item $H'(0)=k$, $H'(x)>0$ and $\lim_{x\nearrow b_r}H''(x) > 0$ (Lemma \ref{lem1}) for $x\geq 0$. Using also Lemma \ref{lem2} (and regularity of $H$) we thus have:
$H'(0)=k$ and $H'(b_r)=1$ with $H'$ decreasing on an interval $[0,c)$ and increasing on $(c,b_r]$ (with $c$ determined by 
$H''(c)=0$). Hence, $H'(x)\leq k$ for $x\geq 0$ and the terms in \eqref{mainTHM-pf3:eq4} are non-negative.
\end{itemize} 
The terms in \eqref{mainTHM-pf3:eqN1} are dealt with in the same way as above. %
Using the same limiting arguments as above we find \ref{mainTHMnew:1a} holds also in the case $b=b_r$.

\emph{Case B:} 
Consider the conditions of \ref{mainTHMnew:2}. Suppose $g$ in  \eqref{mainTHM-pf3:eqN0}--\eqref{mainTHM-pf3:eq4}  is defined as $G$ in \eqref{G-def}. Suppose $b_r\leq b^{*}$.
Observe:
\begin{itemize}
\item $G'(0)\leq k$, $G'(x)>0$ (Lemma \ref{lemmaVA}) and $G'(x) = 1$ for $x \geq b^*$. 
From $G''(b^{*}) = 0$ (easily verified) and Lemma \ref{lem2} follows $G''(x)<0$ for $x <b^{*}$. 
Hence, $G'$ is non-increasing on $[0,b^{*}]$. 
We conclude that $1 \leq G'(x) \leq k$ for  $x\geq 0$. 
%
\item $G(0)=0$ (directly verified) and $G(x)\geq 0$ (by the item above).
\end{itemize}
Using arguments analogous to those above we find that \ref{mainTHMnew:2} holds in the case $b=b^*$. 
Now suppose $b_r>b^{*}$. Observe: 
\begin{itemize}
\item $G'(x)=1$ for $x\geq b_r$ and condition \eqref{Admiss1} imply that the expressions in \eqref{mainTHM-pf3:eq3} vanish (as above).
\item $G'(0)\leq k$, $G'(x)>0$ and $\lim_{x\nearrow b_r}G''(x) > 0$ (Lemma \ref{lem1}) for $x\geq 0$. Using arguments similar to those in the second part of Case A we find $G'(x)\leq k$. 
\item As above we find $G(x)\geq 0$.
\end{itemize}
The usual arguments now imply that \ref{mainTHMnew:2} holds also in the case $b=b_r$.

\emph{Case C:} We have left to prove \ref{mainTHMnew:1b}.  If $G'(0)\leq k$ holds also in this case then the result follows by the exact same arguments as in Case B. Thus, it is enough to show that $G'(0)\leq k$ for $b\geq \hat b$ when \eqref{lemma1:newcondition-k} holds. In \eqref{h-function} we defined the function $h$ by,
\begin{align} 
h(b) = \frac{r_1-r_2}{r_1e^{r_1b} -r_2e^{r_1b}} =  G'(0). 
\end{align}
Hence, $G'(0)= k$ when $b=\hat b$, by definition of $\hat b$ in \eqref{b-hat}. Thus, in order to prove that  $G'(0)\leq k$ for any $b\geq \hat b$ it is enough to show that $h(b)$ is non-increasing in $b$ for $b\geq \hat b$. But $h(b)$ is non-increasing exactly when $b\geq b^*$, as we saw in the proof of Lemma \ref{lemmaVA}. Hence, it is enough to prove that $\hat b \geq b^*$. The right side of \eqref{lemma1:newcondition-k} is equal to $h(b^*)$, cf. \eqref{help-1}. Use this, the definition of $\hat b$ in \eqref{b-hat}, and Lemma \ref{lemmaVR} to find 
$k\leq h(b^*),
k =   h(\hat b),
k\leq h(b^{**}).
$  
But since $h$ is maximal at $b^*$, see the proof of Lemma \ref{lemmaVA}, follows
\begin{align}
k=h(\hat b) \leq h(b^{**})\leq h(b^{*}). \label{help-oct17}
\end{align}
But since $\hat b \in [b^{**},\infty)$ by definition follows that the only possibility is $\hat b \geq b^{*}$; to see this use 
$b^{**}\leq b^{*}$ (Lemma \ref{lemmaVR}),
\eqref{help-oct17}, 
the fact that $h(b)$ is non-decreasing when $b\leq b^*$ and non-increasing when $b\geq b^*$. (Draw a picture).
\end{proof}

 \subsection{Graphical illustrations} \label{graphs-sec}
In this section we consider 
$\mu=0.04$, 
$\sigma^2=0.15$, 
$\alpha=0.05$ and 
$k=1.01$ for which condition \eqref{lemma1:newcondition-k} holds with strict inequality, 
$b^{**}\approx 0.17$, 
$b^* \approx  0.75$ and 
$\hat b \approx  1.58$.    

Recall that $H(x;b_r)$ is the optimal value function without the possibility of bankruptcy, 
$G(x;b_r)$ is the optimal value function without the possibility of capital injection 
and that the optimal value function when both bankruptcy and capital injection is allowed, i.e. $V(x;b_r)$, is for any fixed $b_r$ given by either 
$H(x;b_r)$ or $G(x;b_r)$ according to which is dominating the other, see Corollary \ref{corr1}. 

Figure \ref{fig-habit} illustrates Theorem \ref{mainTHM} and Corollary \ref{corr1} by showing how increasing the dividend payout barrier $b_r$ changes which of $H(x;b_r)$ and $G(x;b_r)$ is dominating. 
Figure \ref{fig-habit} also illustrates Corollary \ref{cor1.5}. 
Figure \ref{fig-habit2} illustrates Corollary \ref{cor2} by showing how $V(x;b_r) = H(x;b_r) \vee G(x;b_r)$ (cf. Corollary \ref{corr1}) decreases in $b_r$. Both figures illustrate that  $H(x;b_r)$ and $G(x;b_r)$ coincide when $b_r=\hat b$, cf. Theorem \ref{mainTHM}.

\begin{figure}[H]
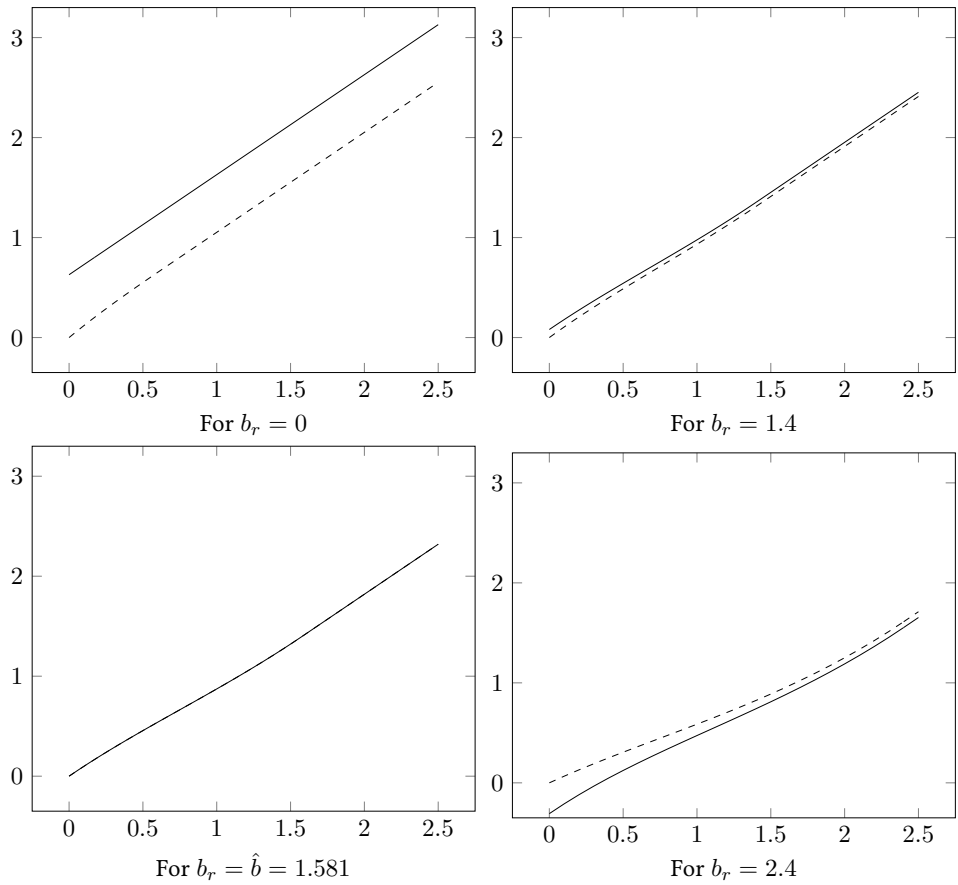


				\caption{
	$x\mapsto H(x;b_r)$ (solid) and $x\mapsto G(x;b_r)$ (dashed) for different values of the dividend payout barrier $b_r$. (Recall that the optimal value function is given by $V(x;b_r) = H(x;b_r) \vee G(x;b_r).)$}\label{fig-habit}
\end{figure}


\begin{figure}[H]
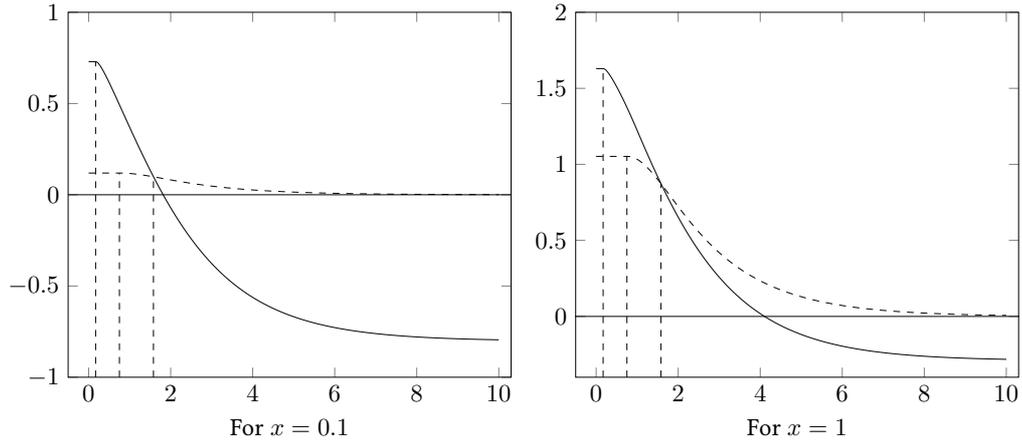

		%

				\caption{
	$b_r\mapsto H(x;b_r)$ (solid) and $b_r\mapsto G(x;b_r)$ (dashed) for different values of initial surplus $x$. 
	The dashed vertical lines in each picture indicate $b^{**}$, $b^{*}$ and $\hat b$, in that order. 
	(Recall that the optimal value function is given by $V(x;b_r) = H(x;b_r) \vee G(x;b_r).)$
	}\label{fig-habit2}
\end{figure}

\section{Conclusions and future research} \label{sec-conclusions}
The main interpretation of the results in the present paper, in particular of items \ref{mainTHMnew:1a} and \ref{mainTHMnew:1b} in Theorem \ref{mainTHM}, see also Figure \ref{fig-habit} and Figure \ref{fig-habit2}, is that if the proportional cost of injecting capital $k$ is low, i.e. if \eqref{lemma1:newcondition-k} holds, then it is optimal to use a double barrier financing strategy and never allow the insurance company to go bankrupt as long as the dividend payout barrier $b_r$ is lower than the level $\hat b$, i.e. the optimal value function is given by $V(x;b_r)=H(x;b_r)$. 
However, if the dividend payout barrier $b_r$ is set higher than $\hat b$ then the optimal behavior switches to an upper barrier strategy that lets the insurance company go bankrupt the first time the surplus reaches zero, i.e. the optimal value function is given by $V(x;b_r)=G(x;b_r)$. 
Moreover, the interpretation of Corollary \ref{cor2}, see also Figure \ref{fig-habit2}, is that an increase in the dividend payout barrier decreases the optimal value function (i.e. the value of the insurance company), with the corresponding limit being zero. 
 
The main economic conclusion of the present paper is that regulation may have the, perhaps unforeseen, effect that if a profitable insurance company 
(corresponding to $\mu>0$) 
has access to a well-functioning financial market 
(corresponding to the proportional costs for capital injection $k$ satisfying \eqref{lemma1:newcondition-k})   
then its owners will inject capital when needed in case the market is unregulated or at least not too heavily regulated ($b_r\leq \hat b$). However, if the regulation is sufficiently heavy ($b_r>\hat b$) then the owners of the same insurance company will change their behavior; specifically, they will never inject capital and instead let the insurance company go bankrupt in the case of financial distress, i.e. in the case of zero surplus.

A potential topic for future research is the investigation of these conclusions for other models considering e.g.:
other surplus processes;   
other types of regulation; additional market frictions, for example fixed costs for capital injection; 
the presence of additional decision variables, for example reinsurance level and investment policy.

\section{Appendix} \label{appendix}

\begin{lem}\label{lem2} Suppose a function $g\in \mathcal C^2$ solves \eqref{ODE} on an interval $[a,b]$, with $g'(x)>0$ for  $x\in  [a,b]$.
\begin{enumerate}[label=(\Roman*)]
\item If $g''(x_0)=0$ for some $x_0 \in (a,b)$ then 
$g''(x) < 0$ for $x \in [a,x_0)$ and 
$g''(x) > 0$ for  $x \in (x_0,b]$.
\item If $g''(b)=0$, then $g''(x) < 0$ for $x \in [a,b)$.
\end{enumerate}
\end{lem}
\begin{proof} If $g$ solves \eqref{ODE} then so does $g'$. Thus, if $g''(x_0)=0$ for some $x_0 \in [a,b]$, then
$(x-x_0)g'(x)g''(x)>0$ for $x\in [a,b],x\neq x_0$, by \cite[Lemma 4.2(b)]{shreve1984optimal}.
The assertions follow.  
\end{proof}

\begin{lem} \label{lem1}  
For $G$ defined in \eqref{G-def} 
and $b=b_r\vee b^{*}$, holds,
\begin{align} 
\left(\mu -\alpha G(x)\right)I_{\{x> b\}} &\leq 0, \enskip \mbox{ for  $x\geq 0$.}\label{lemma1:eq0}
\end{align} 
Moreover,
\begin{align} 
\mbox{if $b_r>b^*$ then } \enskip \lim_{ x \nearrow b }G''(x) &>	0.\label{lemma1:eq1}
\end{align} 
The results also hold for $H$ in \eqref{H-def} when $b^*$ is replaced with $b^{**}$.
\end{lem}
\begin{proof} We directly find
\begin{align} 
\lim_{ x \nearrow b}G''(x) = \frac{r_1^2e^{r_1 b} -r_2^2e^{r_2 b}}{r_1e^{r_1b}-r_2e^{r_2b}}. \label{lem2:pfeq1}
\end{align} 
The numerator in \eqref{lem2:pfeq1} is (as we have seen) strictly positive when $b>b^*$, by definition of $b^*$, see \eqref{b*}. This  proves \eqref{lemma1:eq1}. We also find
\begin{align} 
\lim_{ x \nearrow b}H''(x) = \frac{1}{e^{r_1b}-e^{r_2b}}
\left( r_1(1-ke^{r_2b})e^{r_1 b} -  r_2(1-ke^{r_1b})e^{r_2 b} \right).
\end{align}
Hence, if
\begin{align} 
r_1(1-ke^{r_2b})e^{r_1 b} -  r_2(1-ke^{r_1b})e^{r_2 b}>0, \label{lem2:pfeq2}
\end{align} 
for $b>b^{**}$, then \eqref{lemma1:eq1} holds also for $H$ when replacing $b^{*}$ with $b^{**}$. The inequality \eqref{lem2:pfeq2} is equivalent to
\begin{align} 
r_1(1-ke^{r_2b})e^{r_1 b} -  r_2(1-ke^{r_1b})e^{r_2 b}>0  
& \Leftrightarrow r_1e^{r_1 b} -r_2e^{r_2 b}> k(r_1-r_2)e^{(r_1+r_2)b} \label{appen-b**}\\ 
& \Leftrightarrow r_1e^{-r_2 b} -r_2e^{-r_1 b}> k(r_1-r_2).
\end{align}  
Now, the definition of $b^{**}$ is that the last inequality is an equality when $b=b^{**}$. Hence, if we can show that $r_1e^{-r_2 b} -r_2e^{-r_1 b}$ is (strictly) increasing in $b$ for $b>b^{**}$, then \eqref{lem2:pfeq2} is satisfied for $b>b^{**}$; but this is easily verified using the derivative and \eqref{properties:r1r2}.
We have thus proved \eqref{lemma1:eq1} also for $H$ and $b^{**}$.

Let us prove \eqref{lemma1:eq0}, the same arguments also work for $H$ and $b^{**}$. Since $G$ satisfies \eqref{ODEcondnodiv0} follows,
\begin{align} 
\left(\mu -\alpha G(x)\right)I_{\{x> b\}} & = \left(\mu -\alpha (x-b+G(b))\right)I_{\{x> b\}}\\
& \leq \left(\mu -\alpha G(b)\right)I_{\{x> b\}}.  \label{lem2:pfeq0}
\end{align} 
$G$ also satisfies \eqref{ODE} and $G'(b)=1$. 
Thus, by continuity $\frac{1}{2} \sigma^2 \lim_{ x \nearrow b}G''(x)=\alpha G(b)-\mu$. 
If $b=b^*$ then $\lim_{ x \nearrow b}G''(x)=0$ (to see this use the definition $b^*$ and \eqref{lem2:pfeq1}) and hence $\mu -\alpha G(b)=0$. 
Moreover, if $b>b^*$ follows from \eqref{lemma1:eq1} that $\mu -\alpha G(b)\leq 0 $. Using this in \eqref{lem2:pfeq0} implies that \eqref{lemma1:eq0} holds. 
\end{proof}  

\textbf{Proof of Lemma \ref{lemmaVR}:}
We use \eqref{properties:r1r2} repeatedly.
\ref{lemma1:new:item1.5} is directly verified.

\emph{Proof of \ref{lemma1:new:item1}.} Evaluating \eqref{H-def} at $x=0$ and requiring non-negativity gives
\begin{align} 
\frac{1}{e^{r_1b}-e^{r_2b}} \left( \frac{1-ke^{r_2b}}{r_1}  -\frac{1-ke^{r_1b}}{r_2} \right) \geq 0.
\end{align}
Now simplify. The other case is analogous.

\emph{Proof of \ref{lemma1:new:item2.5}.}
We only prove the first statement (the proof of the second is analogous). The definition of $b^{**}$ in \eqref{b**} is equivalent to,
\begin{align} 
\frac{r_1}{r_2} = \frac{1-ke^{r_1b^{**}}}{1-ke^{r_2b^{**}}}e^{b^{**}(r_2-r_1)} \Leftrightarrow
\frac{1-ke^{r_2b^{**}}}{1-ke^{r_1b^{**}}}=e^{b^{**}(r_2-r_1)}\frac{r_2}{r_1}.
\end{align}  
Now, \eqref{newthm3:cond} is equivalent to 
\begin{align} 
\frac{r_1}{r_2} \geq \frac{1-ke^{r_2b^{**}}}{1-ke^{r_1b^{**}}}.
\end{align}
Thus, \eqref{newthm3:cond} is equivalent to as 
$
\frac{r_1}{r_2} \geq e^{b^{**}(r_2-r_1)}\frac{r_2}{r_1}
$ 
which is equivalent to,
\begin{align} 
\frac{r_1^2}{r_2^2} \leq e^{b^{**}(r_2-r_1)}.\label{lemmaVR0} 
\end{align} 
Solving for $b^{**}$ gives $b^{**} \leq \log(r_2^2/r_1^2)/(r_1-r_2)=b^*$ (cf. \eqref{b*}) and the second equivalence is thus proved. 

To see that \eqref{lemma1:newcondition-k} is equivalent to $b^{**} \leq b^*$ first note that \eqref{lemma1:newcondition-k} holds with equality exactly when $b^*=b^{**}$; to see this note that 
$k = \frac{r_1-r_2}{r_1e^{r_1b^*}-r_2e^{r_2b^*}}$ 
(i.e. \eqref{lemma1:newcondition-k} holds with equality)  
if and only if 
$r_1e^{-r_2b^*}-r_2e^{-r_1b^*} = k(r_1-r_2)$
(i.e. $b^*=b^{**}$, cf. \eqref{b**}), 
which with some effort can be verified by solving for $k$ and using the definition in \eqref{b*}. 
Second, if $k$ decreases then $b^{**}$ decreases; to see this note that the derivative of the left side of \eqref{b**} with respect to $b^{**}$ is positive. It follows that \eqref{lemma1:newcondition-k} is equivalent to $b^{**}\leq b^*$. 

\emph{Proof of \ref{lemma1:new:item2}.} Again we only prove the first statement. In the case $b_r \leq b^{**}$ (i.e. with $b=b^{**}$) holds that $H(0) \geq 0$  is equivalent to \eqref{newthm3:cond}, by item \ref{lemma1:new:item1}. Hence, the result follows from \ref{lemma1:new:item2.5}.

\emph{Proof of \ref{lemma1:new:item4}:} 
By \ref{lemma1:new:item2.5} holds,
\begin{align}
r_1e^{r_1  b^{**}} - r_2e^{r_2 b^{**}} \geq \frac{r_1-r_2}{k}.  
\end{align}

Thus, in the case $b=b^{**}$ (i.e. $b_r\leq b^{**}$) follows, from \ref{lemma1:new:item1}, that $H(0)\leq 0$. Now, if we can prove that $r_1e^{r_1  b} - r_2e^{r_2 b}$ is non-decreasing in $b$, for $b \geq b^{**}$, then follows, from \ref{lemma1:new:item1} that  $H(0)\leq 0 $ also in the case $b_r > b^{**}$ and we are done. Hence, it is enough to show that its derivative, 
$r^2_1e^{r_1b}-r^2_2e^{r_2b}$, is non-negative for $b\geq b^{**}$. But $r^2_1e^{r_1b}-r^2_2e^{r_2b}\geq 0$ is equivalent to $b\geq b^*$(as we have seen) and since $b^{**}\geq b^*$ (by \ref{lemma1:new:item2.5}) follows therefore that $r_1e^{r_1  b} - r_2e^{r_2 b}$ is non-decreasing in $b$, for $b \geq b^{**}$.

\emph{Proof of \ref{lemma1:new:item3}.} \ref{lemma1:new:item2.5} gives 
\begin{align} \label{hhh}
r_1e^{r_1b}-r_2e^{r_2b} \leq  \frac{r_1-r_2}{k} \enskip\mbox{ for } b = b^{**}.
\end{align}
From the proof of Lemma \ref{lemmaVA} we know
$r_1e^{r_1  b} - r_2e^{r_2  b}$ is 
(strictly) increasing in $b$ for $b>b^*$ and 
(strictly) decreasing in $b$ for $b<b^*$; 
moreover, the left side of \eqref{hhh} clearly converges to $\infty$ as $b\rightarrow  \infty$. 
Hence, there exists a unique constant $\hat b \in [b^{**},\infty)$ such that
\begin{align} 
r_1e^{r_1 b}-r_2e^{r_2b} \leq  \frac{r_1-r_2}{k} \mbox{ for } b \leq \hat b, \enskip \mbox{ and } \enskip
r_1e^{r_1 b}-r_2e^{r_2b} \geq  \frac{r_1-r_2}{k} \mbox{ for } b \geq \hat b.
\end{align} 
The result follows from \ref{lemma1:new:item1}.

\textbf{Proof of Corollary \ref{cor1.5}:}
The result is easy to show using the following observations: 

(i) $H(0; \hat b)= G(0; \hat b)=0$ under condition \eqref{lemma1:newcondition-k} (cf. Corollary \ref{corr1}), 

(ii) $G(0; b_r) = 0$ for all $b_r$ (cf. \eqref{G-def}),

(iii) $\hat b>b^{**}$ when \eqref{lemma1:newcondition-k} holds with strict inequality (follows from a direct modification of Lemma \ref{lemmaVR} based on strict inequalities),

(iv) $H(0; b_r)$ is strictly decreasing in $b_r$ when $b_r>b^{**}$ 
(use differentiation, \eqref{properties:r1r2}, the definition of $b^*$ and arguments from the proof of Theorem \ref{newthm1}).

(v) Corollary \ref{corr1}.

\textbf{Proof of Corollary \ref{cor2}:}
From Corollary \ref{corr1} follows that $V(x;b_r) = G(x;b_r)$ for $b_r\geq \hat b$. 
Using \eqref{G-def} and \eqref{properties:r1r2} we directly obtain \ref{cor2:2}.

Let us prove \ref{cor2:1}, \ref{cor2:1a} and \ref{cor2:1b}. 
By Corollary \ref{corr1} and the fact that $b^{**} \leq  b^* \leq \hat b$ when condition \eqref{lemma1:newcondition-k} holds 
(which follows from \ref{lemma1:new:item2.5} in Lemma \ref{lemmaVR} and \emph{Case C} in the proof of Theorem \ref{mainTHM}) 
it suffices to show that for each fixed $x>0$ holds: 

(i)  $G(x;b_r)$ is independent of $b_r$ for $b_r\leq b^*$ and strictly decreasing in $b_r$ for $b_r>b^*$, and 

(ii) $H(x;b_r)$ is independent of $b_r$ for $b_r\leq b^{**}$ and strictly decreasing in $b_r$ for $b_r>b^{**}$. 
 
From \eqref{G-def} we directly see that $G(x;b_r)$ does not depend on $b_r$ for $b_r<b^*$. For $b_r>b^*$ and $0<x<b_r$ it is easy to show that $G(x;b_r)$ is strictly decreasing in $b_r$ (use differentiation and \eqref{properties:r1r2}). This also holds for $b_r>b^*$ and $x>b_r$. Hence, (i) follows from the continuity of $G(x;b_r)$. 
Item (ii) is proved analogously.

\bibliographystyle{abbrv}
\bibliography{kristofferBibl}

\end{document}